\documentclass{amsart}
\usepackage[utf8]{inputenc}
\usepackage[utf8]{inputenc}
\usepackage{hyperref}
\usepackage{amsmath, amsthm, amssymb}
\usepackage{graphicx}
\usepackage{appendix}
\newtheorem{theorem}{Theorem}

\newtheorem{lemma}[theorem]{Lemma}

\newtheorem{proposition}[theorem]{Proposition}

\theoremstyle{definition}

\newtheorem{remark}[theorem]{Remark}
\numberwithin{equation}{section}

\newtheorem{construction}{Construction}

\title{On high-girth expander graphs with localized eigenvectors}

\author{Shohei Satake
}
\address{
Faculty of Advanced Science and Technology, 
Kumamoto University, 2-39-1, Kurokami, Chuo, Kumamoto, 860-8555, Japan}
\email{satakecomb@gmail.com}

\begin{document}

\maketitle

\begin{abstract}
The main purpose of this paper is to construct high-girth regular expander graphs with localized eigenvectors for general degrees, which is inspired by a recent work due to Alon, Ganguly and Srivastava (to appear in Israel J. Math.).

\end{abstract}

\section{Introduction}
For $d \geq 2$, let $G$ be a $(d+1)$-regular graph.
Let $V(G)$ and $E(G)$ denote the vertex and edge set of $G$, respectively.
The {\it girth} of $G$, the length of the shortest cycle of $G$, is denoted by ${\rm girth}(G)$.
The {\it adjacency matrix of $G$}, denoted by $A(G)$, is a square $0$-$1$ matrix with rows and columns indexed by vertices of $G$ in which the $(u, v)$-entry is $1$ if and only if $u$ and $v$ are adjacent in $G$. 
Throughout this paper, we consider normalized eigenvectors of $A(G)$ with respect to $l_2$ norm.

Brooks and Lindenstrauss~\cite{BL2013} proved that for any eigenvector $\mathbf{v}=(v_x)_{x \in V(G)}$ of $A(G)$, real number $\varepsilon\in (0, 1)$ and subset $S \subset V(G)$ such that $||\mathbf{v}_S||^2_2:=\sum_{x \in S}v_x^2\geq \varepsilon$, it holds that
\begin{equation}
\label{eq:BL2013}
    |S| \geq \Omega_d(\varepsilon^2 d^{2^{-7}\varepsilon^2 {\rm girth}(G)}) \; \;\;\;\; (|V(G)|\to \infty).
\end{equation}
In other words, if $A(G)$ has a {\it $(k, \varepsilon)$-localized eigenvector} which is an eigenvector such that there exists $S\subset V(G)$ with $|S|=k$ and $||\mathbf{v}_S||^2_2 \geq \varepsilon$, then it must hold that $k \geq \Omega_d(\varepsilon^2 d^{2^{-7}\varepsilon^2 {\rm girth}(G)})$.
In \cite{GS2019}, Ganguly and Srivastava improved (\ref{eq:BL2013}), that is, under the same assumption, they proved the following inequality.
\begin{equation}
\label{eq:GS2019}
    |S| \geq \frac{\varepsilon d^{(\frac{\varepsilon}{4} \cdot {\rm girth}(G))}}{2d^2}.
\end{equation}
Moreover, for arbitrary $\varepsilon>0$ and $d\geq 2$, the authors of \cite{GS2019} also showed the existence of infinitely many positive integers $m$ and $(d+1)$-regular graphs $G_m$ with $m$ vertices such that ${\rm girth}(G_m)\geq (1/8)\log_d(m)$ and $A(G_m)$ has a $(k, \varepsilon)$-localized eigenvectors with $k=O(d^{4\varepsilon \cdot {\rm girth}(G_m)})$ as $m \to \infty$, which shows that the bound (\ref{eq:BL2013}) is sharp up to the constant $\varepsilon/d^2$.

While the above construction is based on a probabilistic method, Alon, Ganguly and Srivastava~\cite{AGS2020} gave a deterministic construction of $(d+1)$-regular graphs with large girth and localized eigenvectors for every odd prime $d$, improving the lower bound of girth as well as the upper bound of $k$ in the construction in \cite{GS2019}. 
\begin{theorem}[\cite{AGS2020}]
\label{thm-AGS2020}
Let $d$ be an odd prime and $\alpha \in (0, 1/6)$. Then for every $\varepsilon \in (0, 1)$, there exist infinitely many positive integers $m$ and $(d+1)$-regular graphs $G_m$ with $m$ vertices satisfying the following conditions for each $m$.
\begin{enumerate}
    \item ${\rm girth}(G_m)\geq 2\alpha\log_{2d-1}(m) \geq 2\alpha\log_{d}(m) \cdot (1-\frac{\log 2}{\log(2d-1)})$;
    
    \item there exist $\lfloor \alpha \log_d(m) \rfloor$ eigenvalues $\lambda_1, \lambda_2, \ldots, \lambda_{\lfloor \alpha \log_d(m) \rfloor}$ of $A(G_m)$ such that each has a corresponding $(k, \varepsilon)$-localized eigenvectors with $k=O(m^\alpha)$ as $m \to \infty$.
\end{enumerate}
\end{theorem}

The main purpose of this paper is to extend Theorem~\ref{thm-AGS2020} to more general degrees.
The following is the first main theorem in this paper.

\begin{theorem}
\label{thm-main1}
Let $p_j$ denote the $j$-th prime.
Let $d \in [p_j, p_{j+1}]$ be an integer and $\alpha \in (0, 1/6)$.
Suppose that $\frac{\log(p_j)}{\log(p_{j+1})}>6\alpha$.
Then for every $\varepsilon \in (0, 1)$, there exist infinitely many positive integers $m$ and $(d+1)$-regular graphs $G_m$ with $m$ vertices satisfying the same conditions as Theorem~\ref{thm-AGS2020}.
In particular, the same claim holds for every $d\geq p_{10}=29$ and $\alpha \in (0, \frac{\log (p_{10})}{6\log(p_{11})})$ where $\frac{\log (p_{10})}{6\log(p_{11})} \approx 0.163$. 
\end{theorem}

\begin{remark}
For all $d \geq 29$, the last claim of Theorem~\ref{thm-main1} provides infinitely many positive integers $m$ and $(d+1)$-regular graphs $G_m$ with $m$ vertices such that ${\rm girth}(G_m)>0.272 \cdot \log_d(m)$, still improving the lower bound of the girth in the construction provided in \cite{GS2019}, while $\alpha$ must be less than $0.163.<0.166\ldots=1/6$ here.
\end{remark}

Remarkably, the authors of \cite{AGS2020} also proved that for each $m$ and a $(d+1)$-regular graph $G_m$ in Theorem~\ref{thm-AGS2020}, $\lambda(G_m)$, the second largest eigenvalue of $A(G_m)$ in absolute value, is less than $\frac{3}{\sqrt{2}}\sqrt{d} \approx 2.121\sqrt{d}$. This means that Theorem~\ref{thm-AGS2020} also provides near-Ramanujan graphs, that is, regular graphs with near-optimal spectral gap; for details of spectral gaps and Ramanujan graphs, see e.g. \cite{DSV2003}, \cite{M2001} and references therein.

The following second main theorem extends Theorem~\ref{thm-AGS2020} to almost all degrees (see Remark~\ref{rem-primegap}), holding the same bound of the second largest eigenvalue.

\begin{theorem}
\label{thm-main}
Let $d\geq 2$ be an integer and $\alpha \in (0, 1/6)$.
Suppose $d\in [t, p]$ for some odd prime $p$ and positive integer $t<p$ such that $p-t<(\frac{5}{2\sqrt{6}}-1)\sqrt{t} \approx 0.02\sqrt{t}$ and $\frac{\log(t)}{\log(p)}>6\alpha$.
Then, for every $\varepsilon \in (0, 1)$, there exist infinitely many positive integers $m$ and $(d+1)$-regular graphs $G_m$ with $m$ vertices satisfying the same conditions as Theorem~\ref{thm-AGS2020}.
Moreover $\lambda(G_m)\leq \frac{3}{\sqrt{2}} \sqrt{d}$ holds for each $m$.
\end{theorem}

The rest of this paper is organized as follows.
Section~\ref{sect-CM2008} gives a deterministic construction of expander regular graphs with large girth for almost all degrees, which is based on the construction due to Cioab\u{a} and Murty~\cite{CM2008}. Section~\ref{sect-main} proves Theorems~\ref{thm-main1} and \ref{thm-main}.
In Section~\ref{sect-conc}, some concluding remarks are given. 

\section{A construction of expander regular graphs}
\label{sect-CM2008}
This section introduces a construction of expander regular graphs based on the idea due to Cioab\u{a} and Murty~\cite{CM2008}.
The construction is based on the following expander regular graphs constructed by Lubotzky, Philips and Sarnak~\cite{LPS1988}.
\begin{theorem}[\cite{LPS1988}]
\label{thm-LPS1988}
For each odd prime $p$, there exist infinitely many positive integers $n$ and a non-bipartite $(p+1)$-regular graph $R_n$ with $n$ vertices satisfying the following conditions for each $n$.
\begin{enumerate}
    \item $n$ is an even integer;

    \item $\lambda(R_n) \leq 2\sqrt{p}$;

    \item ${\rm girth}(R_n)\geq \frac{2}{3}\log_p(n)$.
\end{enumerate}
\end{theorem}

For our purpose, it is necessary to construct expander regular graphs with more general degrees. 
For each odd prime $p$, take a $(p+1)$-regular graph $R_n$ with $n$ vertices in Theorem~\ref{thm-LPS1988}.
The crucial idea is to delete a {\it $1$-factor}, a $1$-regular spanning subgraph, from $R_n$ to obtain a new $p$-regular graph $R_n'$ with the same vertices.
Note that the existence of a $1$-factor of $R_n$ is confirmed by Theorem~\ref{thm-LPS1988} and the following theorem due to Cioab\u{a}, Gregory and Haemers~\cite{CGH2009}.

\begin{theorem}[\cite{CGH2009}]
\label{thm-CGH2009}
Let $H$ be a $(d+1)$-regular graph. Suppose that $|V(H)|$ is even.
Let $\lambda_3(H)$ denote the third largest eigenvalue of $A(H)$ in absolute value.
Then $H$ has a $1$-factor if
\begin{align}
\label{eq-CGH2009}
     \lambda_3(H) \leq 
     \begin{cases}
    2.85577 & \text{if $d=2$}; \\
    \frac{d-1+\sqrt{(d+1)^2+12}}{2} & \text{if $d$ is odd};\\
    \frac{d-2+\sqrt{(d+2)^2+16}}{2} & \text{if $d$ is even}.
  \end{cases}
\end{align}
\end{theorem}
It follows from the construction of $R'$ that
\begin{align}
\label{eq-lambda}
    \lambda(R_n') \leq 2\sqrt{p}+1,
\end{align}
where this follows from the Weyl's inequality (e.g \cite{CM2008}, \cite{HJ1990}).
Also it follows from the construction of $R_n'$ that
\begin{align}
\label{eq-girth}
    {\rm girth}(R_n')\geq \frac{2}{3}\log_p(n)=\frac{2}{3} \cdot \frac{\log(p-1)}{\log(p)} \cdot \log_{p-1}(n).
\end{align}

The following propositions play an important role in the next section.
\begin{proposition}
\label{prop-seedexpander1}
For each integer $d \in [p_{j}, p_{j+1}]$ with $p_{j+1}-p_j<(1/5)\cdot p_j$, there exist infinitely positive integers $m$ and a $(d+1)$-regular graph $H_n$ with $n$ vertices such that
\begin{equation}
\label{eq-proplambda1}
       \lambda(H_n) \leq  \frac{2}{5}d+2\sqrt{d}, 
\end{equation}
\begin{equation}
\label{eq-propgirth1}
     {\rm girth}(H_n)
    \geq \frac{2}{3} \cdot \frac{\log (d)}{\log (p_{j+1})} \cdot \log_{d}(n)
    > \frac{2}{3} \cdot \frac{\log (p_{j})}{\log (p_{j+1})} \cdot \log_{d}(n).
\end{equation}
\end{proposition}

\begin{proof}
Let $p=p_{j+1}$ and take a $(p+1)$-regular graph $R_n$ in Theorem~\ref{thm-LPS1988}.
Let $d \in [p_j, p_{j+1}]$.
It suffices to prove that the deletion of $p_{j+1}-d$ $1$-factors from $R$ generates a $(d+1)$-regular graph $H_n$ satisfying the conditions of the proposition. 
For $d'\in (d, p_{j+1}]$, suppose the existence of $p_{j+1}-d'$ $1$-factors of $R_n$ and let $H_n'$ be the $(d'+1)$-regular graph obtained by deleting $p_{j+1}-d'$ $1$-factors from $R_n$.
Then, by (\ref{eq-lambda}), 
\begin{align*}
    \lambda(H_n') &\leq 2\sqrt{p_{j+1}}+p_{j}-d' \\
    &=2\sqrt{d'+(p_{j+1}-d')}+p_{j+1}-d'\\
    &\leq 2\sqrt{d'}+2(p_{j+1}-p_j)\\
    &\leq 2\sqrt{d'}+ \frac{2}{5}p_j\\
    &\leq 2\sqrt{d'}+\frac{2}{5}d',
\end{align*}
where the third, fourth and fifth inequalities follow from the assumptions of $d'$, $d$ and $p_j$. 
By Theorem~\ref{thm-CGH2009}, $H_n'$ still has a $1$-factor. 
Thus the induction verifies the existence of $p_{j+1}-d$ $1$-factors of $R_n$.
Since (\ref{eq-proplambda1}) follows from the above discussion, it suffices to prove (\ref{eq-propgirth1}), which is directly obtained by (\ref{eq-girth}).
\end{proof}

\begin{remark}
\label{rem-primegap1}
By the main theorem in \cite{N1952}, it holds that $p_{j+1}-p_j<(1/5)\cdot p_j$ whenever $j\geq 10$, where $p_{10}=29$.
\end{remark}

\begin{proposition}
\label{prop-seedexpander}
Let $p$ be an odd prime and $t\geq 2$ an integer with $0<p-t<(\frac{5}{2\sqrt{6}}-1)\sqrt{t}$.
Then for each integer $d \in [t, p]$, there exist infinitely many positive integers $n$ and a $(d+1)$-regular graph $H_n$ with $n$ vertices such that
\begin{equation}
\label{eq-proplambda}
       \lambda(H_n) \leq \frac{5}{\sqrt{6}}\sqrt{d}, 
\end{equation}
\begin{equation}
\label{eq-propgirth}
     {\rm girth}(H_n)
    \geq \frac{2}{3} \cdot \frac{\log d}{\log p} \cdot \log_{d}(n)
    > \frac{2}{3} \cdot \frac{\log t}{\log p} \cdot \log_{d}(n).
\end{equation}
\end{proposition}

\begin{proof}
As in Proposition~\ref{prop-seedexpander1}, take a $(p+1)$-regular graph $R_n$ in Theorem~\ref{thm-LPS1988}.
The existence of $p-d$ $1$-factors of $R_n$ follows from the same discussion in the proof of Proposition~\ref{prop-seedexpander1}.
Since (\ref{eq-propgirth}) directly follows from (\ref{eq-girth}),
it suffices to prove (\ref{eq-proplambda}), which is verified by (\ref{eq-lambda}) and since it holds that
\begin{align*}
    \lambda(H_n) &\leq 2\sqrt{p}+p-d \\
    &=2\sqrt{d+(p-d)}+p-d\\
    &\leq 2\sqrt{d}+2(p-t)\\
    &\leq 2\sqrt{d}+2\cdot \Bigl(\frac{5}{2\sqrt{6}}-1 \Bigr)\sqrt{d}\\
    &\leq 2\sqrt{d}+\Bigl(\frac{5}{\sqrt{6}}-2 \Bigr)\sqrt{d}
    =\frac{5}{\sqrt{6}}\sqrt{d},
\end{align*}
where the third, fourth and fifth inequalities follow from the assumptions of $d$, $p$ and $t$. 
\end{proof}

\begin{remark}
\label{rem-primegap}
We remark that Proposition~\ref{prop-seedexpander} provides $(d+1)$-regular graphs for almost all positive integers $d$.
Notice that if $p_{j+1}-p_{j}<(\frac{5}{2\sqrt{6}}-1)\sqrt{p_j}$, Proposition~\ref{prop-seedexpander} generates $(d+1)$-regular graphs for all $d \in [p_j, p_{j+1}]$.
In number theory, it is known (\cite{H2019}) that the number of primes $p \leq x$ such that $p=p_j$ with $p_{j+1}-p_{j}>\sqrt{p_{j}}$ is at most $C_{\delta}x^{5/6+\delta}$ as $x \to \infty$ for any $\delta>0$, where $C_{\delta}>0$ depends only $\delta$. 
This implies the desired claim in this remark.
\end{remark}

\section{Proofs of Theorems~\ref{thm-main1} and \ref{thm-main}}
\label{sect-main}
The main task here is to construct $(d+1)$-regular graphs with the conditions of Theorems~\ref{thm-main1} and \ref{thm-main} by using Propositions~\ref{prop-seedexpander1} and \ref{prop-seedexpander}, respectively.

The construction follows from Alon, Ganguly and Srivastava~\cite{AGS2020}.

\begin{construction}
\label{const-AGS2020}
Fix $\alpha \in (0, 1/6)$.
Suppose we have a $(d+1)$-regular graph $H=H_n$ with $n$ vertices provided in Proposition~\ref{prop-seedexpander1} or \ref{prop-seedexpander}.
Let $r=\lfloor \alpha \log_d(n) \rfloor$.
Now construct a $(d+1)$-regular graph $G$ from $H$ by using a {\it $d$-ary tree of depth $r$}, a finite tree such that all vertices except for leafs have degree $d+1$ and every leaf is at distance $r$ from the root. 
First choose a vertex $u \in V(H)$ and take all vertices of $H$ which are at distance at most $r$. 
By the definition of $r$, these vertices induce a $d$-ary tree $T_1$ of depth $r$ with root $u$. 
Let $L_1$ denote the set of leafs of $T_1$ and $V_1=V(T_1) \setminus L_1$.
Take a set of $|L_1|$ vertices, denoted by $L_2$, which are at distance $r+1$ such that there exists a perfect matching $M$ between $L_1$ and $L_2$. 
Then remove the matching $M$ from $H$, and add new $d$-ary tree $T_2$ of depth $r$ to $H$ so that the set of leafs of $T_2$ coincides $L_1$ and other vertices are taken as new vertices outside of $V(H)$.
Here suppose that the graph $G_{1,2}$ on $V(T_1) \cup V(T_2)$ satisfies that
\begin{equation}
\label{eq-pairingtree}
    {\rm girth}(G_{1, 2})\geq 2\log_{2d-1}((d+1)d^{r-1}),
\end{equation}
which is possible by Lemma 2.1 in \cite{AGS2020}.
Finally add another $d$-ary tree $T_3$ of depth $r$ to $H$ so that the set of leafs of $T_3$ coincides $L_2$ and other vertices are taken as new vertices outside of $V(H) \cup V(T_2)$.
The resulting $(d+1)$-regular graph is the desired graph $G$. 
\end{construction}

Note that it holds by the definition of $G$ that
\begin{equation}
\label{eq-vertices}
   m:=|V(G)|=n+2\Bigl\{1+\sum_{1\leq l \leq r-1}(d+1)d^{l-1} \Bigr\}
=n+O(n^{\alpha}).
\end{equation}

To prove Theorems~\ref{thm-main1} and \ref{thm-main}, we have to show that the above construction provides $(d+1)$-regular graphs satisfying the conditions of the theorems. 

First we evaluate the girth of $G$.
\begin{lemma}
\label{lem-girth}
Suppose that $d \in [p_j, p_{j+1}]$ with $\frac{\log (p_{j})}{\log (p_{j+1})}>6\alpha$.
Then, if $n$ is sufficiently large, it holds that
\begin{equation}
\label{eq-lemgirth}
    {\rm girth}(G)\geq 2\log_{2d-1}((d+1)d^{r-1})=2\alpha \log_{2d-1}\Bigl(\Bigl(1+\frac{1}{d} \Bigr)n \Bigr).
\end{equation}
In particular, ${\rm girth}(G)\geq 2\alpha\log_{2d-1}(m)$ if $m$ is sufficiently large.
\end{lemma}

\begin{proof}
By the definition of $G$, in the graph $H \setminus V_1$, the distance between two distinct vertices of $L_1$ is greater than $2r$ since ${\rm girth}(H)>4r$ which follows from the assumption of $\alpha$ and $r$.
The same claim also holds for each pair of two distinct vertices of $L_2$ in the graph $H \setminus V_1$.
The lemma is proved by (\ref{eq-propgirth}), (\ref{eq-pairingtree}) and since it holds that $\frac{2}{3} \cdot \frac{\log (p_j)}{\log (p_{j+1})} \cdot \log_d(n)>4r$, which follows from the assumption that $\frac{\log (p_{j})}{\log (p_{j+1})}>6\alpha$. 
The last claim is directly obtained by (\ref{eq-vertices}).
\end{proof}

Secondly, the following lemma proves the claim about localized eigenvectors.

\begin{lemma}
\label{lem-local}
For each $\varepsilon \in (0, 1)$, Construction~\ref{const-AGS2020} provides $(d+1)$-regular graphs $G$ with $m$ vertices such that
there exist $\lfloor \alpha \log_d(m) \rfloor$ eigenvalues $\lambda_1$, $\lambda_2$, $\ldots$, $\lambda_{\lfloor \alpha \log_d(m) \rfloor}$ of $A(G)$ such that each has a $(k, \varepsilon)$-localized eigenvector with $k=O(m^\alpha)$ as $m \to \infty$.
\end{lemma}

\begin{proof}
The lemma follows from Lemmas 3.1, 3.2 and 3.3 in \cite{GS2019} and Lemma 3.5 in \cite{AGS2020}.
\end{proof}

Now we are ready to prove Theorem~\ref{thm-main1}.
\begin{proof}[Proof of Theorem~\ref{thm-main1}]
The proof is completed by applying Construction~\ref{const-AGS2020} to a $(d+1)$-regular graph $H_n$ obtained from Proposition~\ref{prop-seedexpander1}, and using Lemmas~\ref{lem-girth} and \ref{lem-local}.
The last claim follows from Remark~\ref{rem-primegap1}.
\end{proof}

The remained task is to prove Theorem~\ref{thm-main}.
Now choose parameters $p$, $t$ and $d$ satisfying the assumption of Proposition~\ref{prop-seedexpander}.
Let $H_n$ be a $(d+1)$-regular graph with $n$ vertices in Proposition~\ref{prop-seedexpander} and $G$ a $(d+1)$-regular graph with $m$ vertices obtained from $H_n$ by Construction~\ref{const-AGS2020}.

To prove Theorem~\ref{thm-main}, we have to evaluate $\lambda(G)$.
\begin{lemma}
\label{lem-lambda}
For every fixed $\beta>0$, if $m$ is sufficiently large, then it holds that
\begin{equation}
\label{eq-lemlambda}
    \lambda(G)\leq \Bigl(\frac{3d-1}{\sqrt{d(2d-1)}} + \beta \Bigr)\sqrt{d}.
\end{equation}
In particular, $\lambda(G)\leq \frac{3}{\sqrt{2}}\sqrt{d}$ if $m$ is sufficiently large.
\end{lemma}

We apply the following two lemmas proved in \cite{AGS2020} to prove Lemma~\ref{lem-lambda}.

\begin{lemma}[\cite{AGS2020}]
\label{lem-AGS2020-1}
Let $\mu$ be a non-trivial eigenvalue, an eigenvalue distinct to the top eigenvalue $d+1$, of $A(G)$.
Let $\mathbf{g}=(g_v)_{v \in V(G)}$ be a normalized eigenvector corresponding to $\mu$.
Then it holds that
\begin{equation}
    |\mu|=|\mathbf{g}^TA(G)\mathbf{g}|\leq \lambda(H)+(\sqrt{d}+1)\sum_{u \in L_1 \cup L_2}g_u^2+\frac{2(d+1)|L_1|}{n}.
\end{equation}
\end{lemma}

\begin{lemma}[\cite{AGS2020}]
\label{lem-AGS2020-2}
For any $\beta>0$, suppose that $|\mu|\geq (b+\beta)$ where $b=b(d)=\frac{3d-1}{\sqrt{d(2d-1)}}$. 
Then 
\begin{equation}
    \sum_{u \in L_1 \cup L_2}g_u^2 \to 0 \; \; \; \; \;(m \to \infty).
\end{equation}
\end{lemma}

\begin{proof}[Proof of Lemma~\ref{lem-lambda}]
Let $\mu$ be a non-trivial eigenvalue of $A(G)$.
Suppose that $|\mu|\geq (b+\beta)$.
By (\ref{eq-proplambda}) and Lemma~\ref{lem-AGS2020-1}, it holds that
\begin{align}
\label{eq-lemlambda2}
    |\mu|\leq \frac{5}{\sqrt{6}}\sqrt{d}+(\sqrt{d}+1)\sum_{u \in L_1 \cup L_2}g_u^2+\frac{2(d+1)|L_1|}{n}.
\end{align}
Notice that $|L_1|=(d+1)d^{r-1}=O(n^{\alpha})$, and, by (\ref{eq-vertices}), $m \to \infty$ if and only if $n \to \infty$.
Thus, by (\ref{eq-lemlambda2}) and Lemma~\ref{lem-AGS2020-2}, we have $|\mu|\leq \frac{5}{\sqrt{6}}\sqrt{d}+o(1)$ when $m \to \infty$, which contradicts the assumption that $|\mu|\geq (b+\beta)$ for sufficiently large $m$, since $\frac{5}{\sqrt{6}}=b(2) < b+\beta$ for every $d \geq 2$. 
\end{proof}

\begin{proof}[Proof of Theorem~\ref{thm-main}]
Note that there exists $p(\alpha)$ such that for any $p>p(\alpha)$ and $t<p$ with $p-t<(\frac{5}{2\sqrt{6}}-1)\sqrt{t}$, the condition $\frac{\log (t)}{\log (p)}>6\alpha$ holds.
Thus, by Proposition~\ref{prop-seedexpander} and Construction~\ref{const-AGS2020}, the theorem follows from Lemmas~\ref{lem-girth}, \ref{lem-local} and \ref{lem-lambda}.
\end{proof}

\section{Concluding remarks}
\label{sect-conc}

\begin{enumerate}
    \item[$\cdot$] In analytic number theory, it is conjectured that the order of the magnitude of the gap $p_{j+1}-p_{j}$ can be bounded by a polylogarithmic function of $p_j$ for any sufficiently large $j$. For example, a famous conjecture due to Cram\'{e}r predicts that $p_{j+1}-p_j=O((\log p_j)^2)$; for details, see e.g. \cite{G2004}.
    If this is true, by using Construction~\ref{const-AGS2020}, one can prove that there exists $d_0$ such that for {\it every} $d>d_0$ there exist infinitely many positive integers $m$ and $(d+1)$-regular graphs with $m$ vertices satisfying the all conditions of Theorem~\ref{thm-main}.
    To our best knowledge, the best known (unconditional) bound was  given by Baker, Harman and Pintz, in \cite{BHP2001}, who proved that $p_{j+1}-p_j \leq p_j^{0.525}$ for any sufficiently large $j$.

    
    \item[$\cdot$] From the viewpoint of the study of expander graphs, it would be interesting to improve the bound of $\lambda(G)$. 
    Here we remark that by the same reason noted in Section 1.2 and Remark 4.1 in \cite{AGS2020}, the constructed graph $G$ cannot be a Ramanujan graph which has the optimal second largest eigenvalue. 
\end{enumerate}

\section*{Acknowledgement}
We appreciate Masato Mimura and Hyungrok Jo for their helpful comments.
S. Satake has been supported by Grant-in-Aid for JSPS Fellows 20J00469 of the Japan Society for the Promotion of Science.

\end{document}